\documentclass[11pt]{article}

\usepackage[applemac]{inputenc}
\usepackage[T1]{fontenc}
\usepackage[spanish,english]{babel}
\usepackage{indentfirst}

\usepackage{amsmath,amssymb}
\usepackage{latexsym}
\usepackage{amsfonts}
\usepackage{graphicx}
\usepackage{fancyhdr}
\usepackage[all]{xy}
\usepackage{fullpage}
\usepackage{amsthm}

\newtheorem{thm}{Theorem}[section]
\newtheorem{lemma}[thm]{Lemma}
\newtheorem{prop}[thm]{Proposition}

\theoremstyle{definition}
\newtheorem{defi}[thm]{Definition}

\newtheorem{examples}[thm]{Examples}
\newtheorem{remark}[thm]{Remark}
\newtheorem{notation}[thm]{Notation}

\DeclareMathOperator{\Ext}{Ext}

\DeclareMathOperator{\Hom}{Hom}

\DeclareMathOperator{\image}{Im}

\def\Pn{{\mathbb{P}^n}}

\def\CC{{\mathbb{C}}}

\def\PP{{\mathbb{P}}}

\def\OO{{\mathcal{O}}}
\def\UU{{\mathcal{U}}}

\def\GG{{\mathbb{G}}}

\def\rk{{\mathrm{rk}}}

\begin{document}
\hyphenation{Sch-war-zen-ber-ber}

\title{Schwarzenberger bundles on smooth projective varieties}
\author{Enrique Arrondo, Simone Marchesi and Helena Soares}
\date{}
\maketitle

\abstract{\noindent We define Schwarzenberger bundles on any smooth projective variety $X$. We introduce the notions of jumping pairs of a Steiner bundle $E$ on $X$ and determine a bound for the dimension of its jumping locus. We completely classify Steiner bundles whose set of jumping pairs have maximal dimension, proving that they are all Schwarzenberger bundles.}
\section*{Introduction}

Inspired by a particular family of rank $n$ vector bundles on $\PP^n$ which had been introduced by Schwarzenberger in \cite{Sc}, Steiner bundles on the projective space $\PP^n$ were defined by Dolgachev and Kapranov in \cite{DK}. In this paper, as well as in \cite{V} and \cite{AO},  the authors assign a Steiner bundle on $\PP^n$ to a certain configuration of hyperplanes. Moreover, given a family of special hyperplanes, which Vallès, Ancona and Ottaviani designate by unstable hyperplanes, they are able to reconstruct a Steiner bundle $E$ and determine whether $E$ is a Schwarzenberger bundle.

Schwarzenberger bundles on the projective space of arbitrary rank were recently introduced in \cite{A}. In this work, a generalization of unstable hyperplanes of Steiner bundles, now called jumping hyperplanes, is given. After determining a range for the dimension of the jumping locus, Steiner bundles whose locus has maximal dimension are classified and proved to be all Schwarzenberger. In light of the definition of a Steiner bundle on smooth projective varieties $X$ given in \cite{MS}, the results in \cite{A} were extended in \cite{AM} for the Grassmannian variety.

Our goal in the present paper is to generalize the results in \cite{A} and \cite{AM} to any smooth projective variety.

In Section 1 we recall the definition of Steiner bundles on smooth projective varieties and give an equivalent definition using linear algebra.

In Section 2 we propose a definition of Schwarzenberger bundles on smooth projective varieties (see Definition \ref{defSch}), which generalizes the one given in \cite{A} and \cite{AM}.

In Section 3 we define jumping pair for a Steiner bundle on $X$ (see Definition \ref{defjump}), endow the set of all jumping pairs with the structure of an projective variety and give a lower bound for its dimension.

In Section 4 we obtain an upper bound for the dimension of the jumping variety by studying its tangent space at a fixed jumping pair (see Theorem \ref{thm-uppbound}).

In Section 5 we prove a complete classification of Steiner bundles whose jumping locus has maximal dimension and show that they all are Schwarzenberger bundles (see Theorem \ref{thm-classif}).\vspace{3mm}\\

\noindent \textbf{Acknowledgements}. The three authors were partially supported  by Funda\c{c}\~ao para a Ci\^encia e Tecnologia, project ``Geometria Algébrica em Portugal'', PTDC/MAT/099275/2008; and by Ministerio de Educación y Ciencia de España, project "Variedades algebraicas y analíticas y aplicaciones", MTM2009-06964. The second author is supported by the FAPESP postdoctoral grant number 2012/07481-1. The third author is also partially supported by BRU - Business Research Unit, ISCTE-IUL. Parts of this work were done in UCM-Madrid, IST-Lisbon and UNICAMP-Campinas. The authors would like to thank Margarida Mendes Lopes and Marcos Jardim for the invitation and the warm hospitality.

\section{Steiner bundles on smooth projective varieties}

In this section we recall the definition of Steiner bundles on smooth projective varieties introduced in \cite{MS} and we study some of their properties needed in the sequel.

Let us first fix some notation.

\begin{notation}
We will always work over a fixed algebraically closed field $k$ of characteristic zero and $X$ will always denote a smooth projective variety over $k$.

The projective space $\mathbb{P}(V)$ will be the set of hyperplanes of a vector space $V$ over $k$ or, equivalently, the set of lines in the dual vector space of $V$, denoted by $V^*$.

We will write $\GG(r-1,\PP(V))$ for the Grassmann variety of $(r-1)$-linear subspaces of the projective space $\PP(V)$. This is equivalent to consider $G(r,V^*)$, the set of $r$-dimensional subspaces of the vector space $V^*$.

The dual of a coherent sheaf $E$ on $X$ will be denoted by $E^\vee$.
If $E$ is a vector bundle on $X$ then, for each $x \in X$, $E_x$ is the fibre over $x$.
\end{notation}

In order to define Steiner bundles on a smooth projective variety $X$ we recall the notion of a strongly exceptional pair of coherent sheaves on $X$.
\begin{defi}
Let $X$ be a smooth projective variety. A coherent sheaf $E$ on $X$
is \emph{exceptional} if
\begin{gather*}
\mathrm{Hom}(E,E)\simeq k,  \\
\mathrm{Ext}^{i}(E,E)=0, \text{ for all } i\geq 1.
\end{gather*}
An ordered pair $(E,F)$ of coherent sheaves on $X$ is called an \emph{exceptional pair} if both $E$ and $F$ are exceptional and
\[
\mathrm{Ext}^{p}(F,E)= 0, \text{ for all }p\geq 0.
\]
If, in addition,
\[
\mathrm{Ext}^{p}(E,F)= 0 \text{ for all } p\neq 0,
\]
we say that $(E,F)$ is a \emph{strongly
exceptional pair}.
\end{defi}

\begin{defi}\label{defSteiner}
Let $X$ be a smooth projective variety. An \emph{$(F_0,F_1)$-Steiner bundle} $E$ on $X$ is a vector bundle on $X$ defined by an exact sequence of the form
$$0\to S\otimes F_0\to T \otimes F_1\to E\to 0,$$
where $S$ and $T$ are vector spaces over $k$ of dimensions $s$ and $t$, respectively, and $(F_0,F_1)$ is an ordered pair of vector bundles on $X$ satisfying the two following conditions:
\begin{enumerate}
\item[(i)] $(F_0,F_1)$ is strongly exceptional;
\item[(ii)] $F_0^{\vee}\otimes F_1$ is generated by global sections.
\end{enumerate}
\end{defi}

\begin{examples}\label{exSteiner}\text{ }
\begin{itemize}
\item[(a)] A Steiner bundle according to Dolgachev and Kapranov \cite{DK} is an $(\mathcal{O}_\Pn(-1),\mathcal{O}_\Pn)$-Steiner bundle in the sense of Definition \ref{defSteiner}. More generally, vector bundles $E$ with a resolution of type
    $$0\to \mathcal{O}_\Pn(a)^s\to \mathcal{O}_\Pn(b)^t\to E\to 0,$$
    where $1\leq b-a\leq n$, are $(\mathcal{O}_\Pn(a),\mathcal{O}_\Pn(b))$-Steiner bundles on $\Pn$ (see \cite{MS}).
\item[(b)] Consider the smooth hyperquadric $Q_n\subset\mathbb{P}^{n+1}$, $n\geq 2$, and let $\Sigma_*$ denote the Spinor bundle $\Sigma$ on $Q_n$ if $n$ is odd, and one of the Spinor bundles $\Sigma_+$ or $\Sigma_-$ on $Q_n$ if $n$ is even. The vector bundle $E$ on $Q_n$ defined by an exact sequence of the form
    $$0\to \mathcal{O}_{Q_n}(a)^s\to \Sigma_*(n-1)^t\to E\to 0,$$
    for some $0\leq a\leq n-1$, is an $(\mathcal{O}_{Q_n}(a), \Sigma_*(n-1))$-Steiner bundle (see \cite{MS}).
    \item[(c)] Any exact sequence of vector bundles on the Grassmann variety $\GG:=\GG(r-1,\PP(V))$ of the form
    \[
    0\to \UU^s\to \mathcal{O}_{\GG}^t\to E\to 0,
    \]
    where $\UU$ denotes the rank $r$ universal subbundle of $\GG$, defines a $(U, \mathcal{O}_{\GG})$-Steiner bundle $E$ on $\GG$. These bundles were studied by Arrondo and Marchesi in \cite{AM}.
    \item[(d)]
Let $X = \tilde{\PP}^2 \backslash \{p_1,p_2,p_3\}$ be the blow up of $\PP^2$ at three points $p_1$, $p_2$ and $p_3$. Let $K_X=-3L+E_1+E_2+E_3$ denote the canonical divisor, where $L$ is the divisor corresponding to a line not passing through any of the three points and $E_i$ is the exceptional divisor of the blow up at the point $p_i$, $i=1,2,3$.
Take $F_0 = -2L + E_1 + E_2 + E_3$ and $F_1 = \OO_X$.\\
Notice that $H^0(F_0^\vee)$ is globally generated and is the set of conics that pass through three points. In particular, $h^0(F_0^\vee) = 3 = \dim X +1$.

Let us now prove that the pair of vector bundles $(F_0,F_1)$ is strongly exceptional.
Since both $F_0$ and $F_1$ are line bundles on a projective variety, the fact that $\Hom(F_0,F_0) =\Hom(\OO_X,\OO_X) =  \CC$ and $\Ext^i(F_0,F_0)=\Ext^i(\OO_X,\OO_X)=0$, $i=1,2$, is straightforward.

Using Riemann-Roch formula we obtain $\chi(F_0^\vee) = 3$ and thus $h^1(F_0^\vee) = h^2(F_0^\vee)$. Since $H^2(F_0^\vee) = H^0(K_X - F_0) = H^0 (-L) = 0$, we get that $\Ext^i(F_0,\OO_X) = H^i(F_0^\vee)=0$, for $i=1,2$.

From the fact that $H^0(F_0^\vee)\neq 0$ and $\Hom(F_0,F_0)=H^0(F_0\otimes F_0^\vee)\neq 0$ it follows that $\Hom(\OO_X,F_0)$ must be trivial. Furthermore, $\Ext^2(\OO_X,F_0)=H^2(F_0)=H^0(-F_0+K_X)=H^0(-L)=0$. Then also $\Ext^1(\OO_X,F_0)=H^1(F_0)=0$ because one can check that $\chi(F_0) = 0$.

So, we have just proved that any vector bundle $E$ fitting in a sequence of type
\[
0\to F_0^s\to \OO_X^t\to E\to 0
\]
is an $(F_0,\OO_X)$-Steiner bundle on the blow up $X$.
\end{itemize}
\end{examples}

The following proposition gives a characterization of $(F_0,F_1)$-Steiner bundles on a smooth projective variety $X$ by means of linear algebra (recall also Lemma 1.2 in \cite{A} or Lemma 1.7 in \cite{AM}). This interpretation will play an essential role for studying Schwarzenberger bundles on $X$.

\begin{prop}\label{linalg}
To give an $(F_0,F_1)$-Steiner bundle on a smooth projective variety $X$
$$0\to S\otimes F_0\to T \otimes F_1\to E\to 0,$$
is equivalent to give a linear map $\varphi:T^*\to S^*\otimes H^0(F_0^{\vee}\otimes F_1)$ such that, for each $x\in X$,
the induced linear map
$$\widetilde{\varphi}_x:T^*\otimes (F_1)_x^*\to S^*\otimes {(F_0^\vee)}_x$$
is surjective.
\end{prop}

\begin{proof}
Dualizing the sequence defining the Steiner bundle $E$, we see that to give a map $S\otimes F_0\to T \otimes F_1$ is the same as to give a map $\widetilde{\varphi}:T^*\otimes F_1^{\vee}\to S^*\otimes F_0^{\vee}$.

Twisting by $F_1$, taking cohomology and using condition (ii) of Definition \ref{defSteiner}, this is clearly equivalent to a linear map $\varphi:T^*\to S^*\otimes H^0(F_0^{\vee}\otimes F_1)$ with fibers $\varphi_x:T^*\to S^*\otimes H^0({(F_0^\vee)}_x\otimes (F_{1})_x)\cong S^*\otimes {(F_0^\vee)}_x\otimes (F_{1})_x$. Hence, $\varphi_x$ induces a linear map $\widetilde{\varphi}_x:T^*\otimes (F_1)_x^*\to S^*\otimes {(F_0^\vee)}_x$ and, moreover, the map $S\otimes F_0\to T \otimes F_1$ is injective if and only if $\widetilde{\varphi}_x$ is surjective for each $x\in X$.

\end{proof}

In what follows, $\varphi$ will always denote the linear map associated to an $(F_0,F_1)$-Steiner bundle introduced in Proposition \ref{linalg}.

\begin{lemma}\label{deflemma}
Let $E$ be an $(F_0,F_1)$-Steiner bundle on a smooth projective variety $X$. Then the following properties are equivalent:
\begin{itemize}
\item[(i)] $\varphi$ is injective.
\item[(ii)] $H^0(E^{\vee}\otimes F_1)=0$.
\item[(iii)] $E$ cannot split as $E_K\oplus(K^*\otimes F_1)$, where $K\subset \ker\varphi\subset T^*$ is a vector space.
\end{itemize}
Moreover, if $\varphi$ is not injective then $E_K$ is the $(F_0,F_1)$-Steiner bundle corresponding to the map $T^*/K\to S^*\otimes H^0(F_0^\vee\otimes F_1)$. In particular, when $K=\ker\varphi$, there is a splitting $E=E_{\ker\varphi}\oplus ((\ker\varphi)^*\otimes F_1)$, $E_{\ker\varphi}$ is the $(F_0,F_1)$-Steiner bundle corresponding to the inclusion $\image\varphi\hookrightarrow S^*\otimes H^0(F_0^\vee\otimes F_1)$ and $H^0(E_{\ker\varphi}^\vee\otimes F_1)=0$.
\end{lemma}

\begin{proof}
To see that $(i)$ is equivalent to $(ii)$, it is enough to observe that, after dualizing and twisting by $F_1$ the exact sequence defining $E$, we get a short exact sequence
$$0\to E^\vee \otimes F_1\to (F_1^\vee)^t\otimes F_1\to (F_0^\vee)^s\otimes F_1\to 0.$$
So, taking cohomology and using the fact that $(F_0,F_1)$ is an exceptional pair, we see that $\varphi$ is injective if and only if $H^0(E^{\vee}\otimes F_1)=0$.

If $E\cong E_K\oplus (K\otimes F_1)$ then there is a non-trivial morphism $E\to F_1$, i.e. $H^0(E^{\vee}\otimes F_1)\neq 0$. This proves that $(i)$ implies $(iii)$.

Reciprocally, suppose $\varphi$ is not injective and let $0\neq \tilde{t}\in \ker\varphi\subset T^*$. There is a commutative triangle of linear maps
\[
\xymatrix{
T^* \ar[d] \ar[r]^(.3){\varphi} &  S^* \otimes H^0(F_0^\vee\otimes F_1)  \\
T^*/\tilde{t} \ar[ur]_{\varphi'}& }
\]
which, by Proposition \ref{linalg}, induces a commutative triangle
\[
\xymatrix{
T^*\otimes (F_1)_x^* \ar[d] \ar@{>>}[r]^(.5){\widetilde{\varphi}_x} &  S^*\otimes {(F_0^\vee)}_x  \\
T^*/\tilde{t} \otimes (F_1)_x^* \ar@{>>}[ur]_{\widetilde{\varphi'}_x}& }
\]
In particular, we see that $\varphi'$ is a linear map such that
$\varphi'_x$
induces a surjective linear map
$\widetilde{\varphi'}_x:T^*\otimes (F_1)_x^*\to S^*\otimes {(F_0^\vee)}_x$. Therefore, $\varphi'$ defines an $(F_0,F_1)$-Steiner bundle $E_{\tilde{t}}$ and we have a commutative diagram
\[
\xymatrix{ &  & 0 \ar[d] & & \\
0\ar[r] & S\otimes F_0 \ar[r] \ar@{=}[d] & \left(T^*/\tilde{t} \right)^*\otimes F_1 \ar[r] \ar[d] & E_{\tilde{t}} \ar[r] \ar[d] & 0 \\
0\ar[r] & S\otimes F_0 \ar[r]  & T\otimes F_1 \ar[r] \ar[d] & E\ar[r] & 0 \\
&  & \langle \tilde{t}\rangle^*  \otimes F_1\cong F_1 \ar[d] & & \\
&  & 0  &   &
}
\]
From snake's lemma we immediately deduce that the morphism $E_{\tilde{t}}\to E$ is injective and that its cokernel is isomorphic to $F_1$. Hence, the diagram above can be completed as follows:
\[
\xymatrix{ &  & 0 \ar[d] & 0 \ar[d] & \\
0\ar[r] & S\otimes F_0 \ar[r] \ar@{=}[d] & \left(T^*/\tilde{t} \right)^*\otimes F_1 \ar[r] \ar[d] & E_{\tilde{t}} \ar[r] \ar[d] & 0 \\
0\ar[r] & S\otimes F_0 \ar[r]  & T\otimes F_1 \ar[r] \ar[d] & E\ar[r] \ar[d] & 0 \\
&  & \langle \tilde{t}\rangle^*  \otimes F_1\cong F_1 \ar[d] \ar@{=}[r] & F_1 \ar[d] & \\
&  & 0  & 0  &
}
\]
Now, applying the functor $\mathrm{Hom}(F_1,-)$ to the exact sequence that defines $E$ (the middle row in the diagram), it follows that $\mathrm{Hom}(F_1,E)\cong k^t$. Applying the same functor to the sequence in the right column, we get that $\mathrm{Ext}^1(F_1,E_{\tilde{t}})=0$ and thus, $E$ splits as $E_K\oplus (K^*\otimes F_1)$, where, by construction, $K=\langle \tilde{t}\rangle^*\subset \ker\varphi\subset T^*$.

The last statements are an immediate consequence of Proposition \ref{linalg} and the equivalences just proved.
\end{proof}

The previous lemma motivates the following definition.

\begin{defi}
An $(F_0,F_1)$-Steiner $E$ is called \emph{reduced} if one of the properties in Lemma \ref{deflemma} hold. The $(F_0,F_1)$-Steiner bundle $E_0:=E_{\ker\varphi}$ is called the \emph{reduced summand} of $E$.
\end{defi}

\begin{examples}\label{exSteinerRed}\text{ }
\begin{itemize}
\item[(a)] Any $(\mathcal{O}_\Pn(a),\mathcal{O}_\Pn(b))$-Steiner bundle on $\Pn$ of rank $n$ is reduced. If this was not the case, by the previous lemma there would exist an $(\mathcal{O}_\Pn(a),\mathcal{O}_\Pn(b))$-Steiner bundle $E_{\ker\varphi}$ of rank less than $n$, contradicting Proposition \ref{thm-rank}.
\item[(b)] Let $s\leq k+1$. Then there is an exact sequence on the Grassmannian variety,
    \[
    0\to \UU^s\to \mathcal{O}_{\GG}^t\to Q^s\oplus\OO^\alpha_\GG\to 0,
    \]
    where $t=s(n+1)+\alpha$, with $\alpha\geq 0$, and $Q$ is the quotient bundle on $\GG$. If $\alpha>0$ then $Q^s\oplus\OO^\alpha_\GG$ is a $(\UU, \mathcal{O}_{\GG})$-Steiner bundle on $\GG$ that is not reduced (see \cite{AM}).
\end{itemize}
\end{examples}
\bigskip

From now on, we will restrict our study to $(F_0,\OO_X)$-Steiner bundles on a smooth projective variety $X$. In particular, $F_0^\vee$ will be a globally generated vector bundle on $X$. Denote $f_0=\rk(F_0)$.

Consider the exact sequence of vector bundles on $X$ induced by the fact that $F_0^\vee$ is generated by its global sections
\begin{equation}\label{vbK}
0\longrightarrow F_0 \longrightarrow H^0(F_0^\lor)^* \otimes \OO_X \longrightarrow Q \longrightarrow 0.
\end{equation}
The following proposition determines a range on the rank of any $(F_0,\OO_X)$-Steiner bundle on $X$ which is a necessary condition for its existence.

\begin{prop}\label{thm-rank}
 Let $E$ be an $(F_0,\OO_X)$-Steiner bundle on a projective variety $X$ such that each Chern class of $S\otimes Q$ is non-zero, where $Q$ is as defined in (\ref{vbK}). Then
$$
\rk(E) \geq \min (\dim X, s \cdot\rk(Q)).
$$
\end{prop}
\begin{proof}
Let $\alpha$ be the vector bundle morphism defined by
$$
S \otimes F_0 \stackrel{\alpha}{\longrightarrow} T \otimes \OO_X \dashrightarrow E.
$$
Denote by $r = \rk (E) = t - s f_0$ and consider the \emph{degeneracy locus} of $\alpha$,
$$
D_\alpha=\left\{x \in X \:|\: \rk (\alpha_x) \leq s f_0 -1 \right\},
$$
i.e. the set of
points $x$ of $X$ such that the rank of the morphism $\alpha_x$ is not maximal. We know that $E$ is Steiner if and only if $\alpha$ is injective in each fiber, that is, if and only if $D_\alpha=\emptyset$.

Using Porteous' formula we get that the expected codimension of the degeneracy locus is equal to
$$
t - s f_0 +1 = \rk (E) +1 = r+1.
$$

Hence, it is clear that when $r+1 > \dim X$, or equivalently, when $\rk (E) \geq \dim X$, we can ensure that $\alpha$ is injective.

Otherwise, when $r+1\leq \dim X$, the degeneracy locus will be empty if and only if its fundamental class, given by the Chern class $c_{r+1}(S\otimes Q)$, is zero. From our hypothesis on the Chern classes of $S\otimes Q$, this holds if and only if $r+1 > \rk(S\otimes Q)$, i.e $\rk(E) \geq s\cdot \rk(Q)$.
\end{proof}

The previous result tells us that the family of Steiner bundles is naturally divided in two subfamilies, depending on whether $\dim X < s\cdot \rk(Q)$ or $\dim X \geq s\cdot \rk(Q)$. The latter is completely classified in the next proposition.

\begin{prop}\label{prop-trivial}
Let $E$ be an $(F_0,\OO_X)$-Steiner bundle on a smooth projective variety $X$ such that each Chern class of $S\otimes Q$ is non-zero, where $Q$ is as defined in (\ref{vbK}). If $\dim X \geq s \cdot\rk(Q)$, then $E \simeq (S \otimes Q) \oplus \OO_X^p$, for some $p \geq 0$.
\end{prop}

\begin{proof}
Suppose $E$ is reduced and consider the following diagram:
$$
\xymatrix{
& 0 \ar[d] & 0 \ar[d] \\
0 \ar[r] & E^\lor \ar[d]^{\xi} \ar[r] & T^* \otimes \OO_X \ar[d]^{\varphi \otimes id_X} \ar[r] & S^* \otimes F_0^\lor \ar[d]^{\simeq} \ar[r] & 0 \\
0 \ar[r] & S^* \otimes Q^\vee \ar[d] \ar[r] & S^* \otimes H^0(F_0^\lor) \otimes \OO_X \ar[d] \ar[r] & S^* \otimes F_0^\lor \ar[r] & 0 \\
& \frac{S^* \otimes H^0(F_0^\lor)}{T^*} \otimes \OO_X \ar[d] \ar[r]^{\simeq} & \frac{S^* \otimes H^0(F_0^\lor)}{T^*} \otimes \OO_X \ar[d] \\
& 0 & 0
}
$$

Since $\varphi$ is injective, we have $t\leq sh^0(F_0^\vee)$. From the hypothesis on the dimension of $X$, it follows from Proposition \ref{thm-rank} that $\rk(E)\geq s\cdot \rk(Q)$ and hence $t\geq sh^0(F_0^\vee)$. Therefore, $\varphi \otimes id_X$ is an isomorphism and so is $\xi$. Thus $E\simeq S\otimes Q$.

When $E$ is not reduced there is a splitting $E=E_{K}\oplus (K^*\otimes \OO_X)$ (recall Lemma \ref{deflemma}), for some $K\subset \ker\varphi\subset T^*$, where $E_K$ is a reduced $(F_0,\OO_X)$-Steiner bundle. So by the previous argument, $E_K$ is isomorphic to $S \otimes Q$ and $E$ is isomorphic to $(S \otimes Q) \oplus \OO_X^p$.
\end{proof}

\begin{remark}\label{gen-fin}
Consider the natural morphism $\sigma : X \longrightarrow \GG(f_0-1, \PP(H^0(F_0^\lor)))$. If $\sigma$ is generically finite then each Chern class of $S\otimes Q$ is non-zero.
Conversely, if each Chern class of $S\otimes Q$ is non-zero and $\dim X\leq s\cdot\rk(Q)$ then $\sigma$ is generically finite.
Both statements can be proved using the projection formula.
\end{remark}

Having classified all Steiner bundles in the case when $\dim X\geq s\cdot\rk(Q)$, from now on we will always suppose that $\dim X< s\cdot\rk(Q)$. In the sequel we will refer to the bundles $E \simeq (S \otimes Q) \oplus \OO_X^p$ as the \emph{trivial Steiner bundles} on $X$.

\section{Generalized Schwarzenberger on smooth projective varieties}

Our goal in this section is to generalize Schwarzenberger bundles on the projective space and on the Grassmann variety $\GG(k,n)$, as defined in \cite{A} and \cite{AM}, respectively, to any smooth projective variety $X$.

We first recall Schwarzenberger bundles on $\GG(k,n)$, following the construction in \cite{AM}.
Let us consider two globally generated vector bundles $L,M$ over a projective variety $Y$, with $h^0(M) = n+1$ and with the identification $\PP^n = \PP(H^0(M)^*)$. Consider the composition
$$
H^0(L) \otimes \UU \longrightarrow H^0(L) \otimes H^0(M) \otimes \OO_\GG \longrightarrow H^0(L\otimes M) \otimes \OO_\GG.
$$
We want this composition to be injective, that is we want it be injective in each fiber. This is equivalent to fixing $k+1$ independent global sections $\{\sigma_1,\ldots,\sigma_{k+1}\}$ in $H^0(M)$ in correspondence to the point $\Gamma = [<\sigma_1,\ldots,\sigma_{k+1}>] \in \GG(k,n)$ and requiring the injectivity of the following composition
$$
H^0(L) \otimes <\sigma_1,\ldots,\sigma_{k+1}> \longrightarrow H^0(L) \otimes H^0(M) \longrightarrow H^0(L\otimes M),
$$
given by multiplication with the global section subspace $<\sigma_1,\ldots,\sigma_{k+1}>$.

If the injectivity holds for each point of the Grassmannian then a \emph{Schwarzenberger bundle $F= F(Y,L,M)$} is defined as an $(\UU,\OO)$-Steiner bundle defined by the resolution
$$
0 \longrightarrow H^0(L) \otimes \UU \longrightarrow H^0(L\otimes M) \otimes \OO_\GG \longrightarrow F \longrightarrow 0.
$$

\bigskip

Motivated by the previous case, we will now describe the the construction of a Schwarzenberger bundle on a smooth projective variety $X$.
Let $Z$ be a projective variety and consider a non-degenerate linearly normal morphism $\psi: Z\to G(f_0,H^0(F_0^\vee))$. Consider the composition
$$0\to H^0(L)\otimes F_0\to H^0(L)\otimes H^0(F_0^\vee)^*\otimes \mathcal{O}_X\to H^0(L)\otimes H^0(\psi^*\UU^\vee)\otimes \mathcal{O}_X\to H^0(L\otimes\psi^*\UU^\vee)\otimes \mathcal{O}_X,$$
where $\UU$ denotes the universal subbundle on $G:=G(f_0,H^0(F_0^\vee))$, and $L$ is a globally generated locally free sheaf on $Z$.
The first map is given by the monomorphism $F_0\hookrightarrow H^0(F_0^\vee)^*\otimes \mathcal{O}_X$ (recall that $F_0^\vee$ is generated by global sections and hence there is an epimorphism $H^0(F_0^\vee)\otimes \mathcal{O}_X\twoheadrightarrow F_0^\vee$). The second map is just given by the fact that $H^0(F_0^\vee)^*\cong H^0(\UU^\vee)\cong H^0(\psi^*\UU^\vee)$, and the last map is the one induced by the natural morphism $H^0(L)\otimes H^0(\psi^*\UU^\vee)\to H^0(L\otimes\psi^*\UU^\vee)$.

Let us show that this composition is injective, i.e. that
$$\eta: H^0(L)\otimes F_0\to H^0(L\otimes\psi^*\UU^\vee)\otimes \mathcal{O}_X$$
is injective on each fiber. Given any $x\in X$, the composition $(F_0)_x\to H^0(\psi^*\UU^\vee)$ of the first two maps (on the second factor) is obviously injective. Observing that to give a morphism $(F_0)_x\hookrightarrow H^0(\psi^*\UU^\vee)$ is the same as to give a map $(F_0)_x\otimes\mathcal{O}_Z\hookrightarrow \psi^*\UU^\vee$, we deduce that $L\otimes (F_0)_x\hookrightarrow L\otimes\psi^*\UU^\vee$ is still injective. Finally, applying cohomology, we conclude that $\eta_x:H^0(L)\otimes (F_0)_x\hookrightarrow H^0(L\otimes\psi^*\UU^\vee)$ is injective.

Therefore, we have just constructed an $(F_0,\OO_X)$-Steiner bundle on $X$ defined by
$$0\to H^0(L)\otimes F_0\to H^0(L\otimes\psi^*\UU^\vee)\otimes \mathcal{O}_X \to E\to 0.$$
Furthermore, observe that, under the identification $H^0(F_0^\vee)^*\cong H^0(\psi^*\UU^\vee)$, the map $\varphi:H^0(L\otimes\psi^*\UU^\vee)^*\to H^0(L)^*\otimes H^0(F_0^\vee)$ is nothing but the dual of the multiplication map $H^0(L)\otimes H^0(\psi^*\UU^\vee)\to H^0(L\otimes\psi^*\UU^\vee)$.

This construction allows us to generalize the notion of a Schwarzenberger bundle to any smooth projective variety.

\begin{defi}\label{defSch}
Let $X$ be a smooth projective variety. Let $Z$ be a projective variety provided with a non-degenerate linearly normal morphism $\psi: Z\to G(f_0,H^0(F_0^\vee))$ and $L$ a globally generated locally free sheaf on $Z$. A \emph{$(Z,\psi,L)$-Schwarzenberger bundle} on $X$ is the $(F_0,\mathcal{O}_X)$-Steiner bundle $E$ defined by the short exact sequence
$$0\to H^0(L)\otimes F_0\to H^0(L\otimes\psi^*\UU^\vee)\otimes \mathcal{O}_X \to E\to 0,$$
constructed above.
\end{defi}

\begin{remark} Observe that a Schwarzenberger bundle $F= F(Y,L,M)$ on $\GG(k,n)$ is a $(Z,\psi,L)$-Schwarzenberger bundle in the sense of Definition \ref{defSch}, where $Z=Y$ and $\psi:Y\to \mathbb{P}^n=\mathbb{P}(H^0(M)^*)$.
\end{remark}

\section{Jumping pairs of Steiner bundles}
We would like to know when an $(F_0,\OO_X)$-Steiner bundle on a smooth projective variety $X$ is a $(Z,\psi,L)$-Schwarzenberger bundle on $X$. In order to answer to this question we will look for some property that distinguishes a Schwarzenberger bundle. More precisely, we will show that for any point $z\in Z$, a Schwarzenberger bundle on $X$ associates a special subspace of $H^0(F_0^\vee)$.

If $E$ is a $(\psi,Z,L)$-Schwarzenberger bundle on $X$, we have
$$S=H^0(L), \quad T=H^0(L\otimes\psi^*U^\vee),$$
and we already observed that the map $\varphi:H^0(L\otimes\psi^*U^\vee)^*\to H^0(L)^*\otimes H^0(F_0^\vee)$ is the dual of the multiplication map. Suppose $\rk(L)=a$.

For each $z\in Z$, the surjective morphisms $H^0(\psi^*U^\vee)\twoheadrightarrow (\psi^*U^\vee)_z$ and $H^0(L)\twoheadrightarrow L_z$ induce, respectively, an $f_0$-dimensional subspace $(\psi^*U)_z\subset H^0(\psi^*U^\vee)^*\cong H^0(F_0^\vee)$ and a subspace $L_z^*\subset H^0(L)^*$ of dimension $a$. Since $\varphi$ maps $H^0(L_z\otimes(\psi^*U^\vee)_z)^*$ isomorphically into $H^0(L_z)^*\otimes H^0((\psi^*U^\vee)_z)^*\cong L_z^*\otimes (\psi^*U)_z$, each point $z\in Z$ yields a pair of subspaces $(L_z^*,(\psi^*U)_z)$ such that $L_z\otimes (\psi^*U)_z\in \image\varphi$.

This property of the Schwarzenberger bundles motivates the following definition of jumping subspaces and jumping pairs of an $(F_0,\mathcal{O}_X)$-Steiner bundle on $X$.

\begin{defi}\label{defjump}
Let $E$ be an $(F_0,\mathcal{O}_X)$-Steiner bundle on $X$. An \emph{$(a,b)$-jumping subspace of $E$} is a $b$-dimensional subspace $B\subset H^0(F_0^\vee)$ for which there exists an $a$-dimensional subspace $A\subset S^*$ such that $A\otimes B$ is in the image $T^*_0$ of $\varphi: T^*\to S^*\otimes H^0(F_0^\vee)$. Such a pair $(A,B)$ is called an \emph{$(a,b)$-jumping pair of $E$}.

We will write $J_{a,b}(E)$ and $\tilde{J}_{a,b}(E)$ to denote, respectively, the set of $(a,b)$-jumping subspaces and the set of $(a,b)$-jumping pairs of $E$. We will also write $\Sigma_{a,b}(E)$ to denote the set of $a$-dimensional subspaces $A\subset S^*$ for which there exists a $b$-dimensional subspace $B\subset H^0(F_0^\vee)$ such that $(A,B)$ is a $(a,b)$-jumping pair of $E$.
\end{defi}

It turns out that the set of jumping pairs has a geometric interpretation similar to the one obtained in Lemma 2.4. in \cite{A} that endows $\tilde{J}_{a,b}(E)$ with a natural structure of projective variety. Consider the natural generalized Segre embedding
$$\nu:G(a,S^*)\times G(b,H^0(F_0^\vee))\to G(ab,S^*\otimes H^0(F_0^\vee))$$
given by the tensor product of subspaces.
Then, we can state the following result (the proof is essentially the same as in Lemma 2.4 in \cite{A}, so we omit it).
\begin{lemma}\label{lem-class}
Let $E$ be an $(F_0,\mathcal{O}_X)$-Steiner bundle on $X$. Then:
\begin{itemize}
\item[(i)] the set $\tilde{J}_{a,b}(E)$ of jumping pairs of $E$ is the intersection of the image of $\nu$ with the subset $G(ab,T^*_0)\subset G(ab,S^*\otimes H^0(F_0^\vee))$, i.e.
$$\tilde{J}_{a,b}(E)=\image\nu\cap G(ab,T^*_0).$$
\item[(ii)] If $\pi_1$ and $\pi_2$ are the respective projections from $\tilde{J}_{a,b}(E)$ to $G(a,S^*)$ and $G(b,H^0(F_0^\vee))$, then $\Sigma_{a,b}(E)=\pi_1(\tilde{J}_{a,b}(E))$ and $J_{a,b}(E)=\pi_2(\tilde{J}_{a,b}(E))$.
\item[(iii)] Let $\mathcal{A}$, $\mathcal{B}$ and $\mathcal{Q}$ be the universal quotient bundles of respective ranks $a$, $b$ and $ab$ of $G(a,S^*)$, $G(b,H^0(F_0^\vee))$ and $G(ab,T_0^*)$. Assume that the natural maps
\begin{gather*}
\alpha: H^0(G(a,S^*),\mathcal{A})\to H^0(\tilde{J}_{a,b}(E),\pi^*_1\mathcal{A}) \\
\beta: H^0(G(b,H^0(F_0^\vee)),\mathcal{B})\to H^0(\tilde{J}_{a,b}(E),\pi^*_2\mathcal{B}) \\
\alpha: H^0(G(ab,T_0^*),\mathcal{Q})\to H^0(\tilde{J}_{a,b}(E),\mathcal{Q}_{|\tilde{J}_{a,b}(E)})
\end{gather*}
are isomorphisms. Then the reduced summand $E_0$ of $E$ is the $(\tilde{J}_{a,b}(E), |\pi_2^* \mathcal{B}|,\pi_1^*\mathcal{A})$-Schwarzenberger bundle.
\end{itemize}
\end{lemma}

We can also deduce the following:
\begin{lemma}
Let $E$ be an $(F_0,\mathcal{O}_X)$-Steiner bundle on $X$. Then
$$\tilde{J}_{a,b}(E)=\tilde{J}_{a,b}(E_0),$$
where $E_0$ is the reduced summand of $E$. In particular, $J_{a,b}(E)=J_{a,b}(E_0)$.
\end{lemma}
\begin{proof}
In Lemma \ref{deflemma} we saw that $E=E_0\oplus (\ker\varphi)^*\otimes F_1$ and that $E_0$ corresponds to the linear map $\varphi':T^*/\ker\varphi\to S^*\otimes H^0(F_0^\vee)$. The statement now follows immediately, since $T_0^*=\image\varphi=\image\varphi'$.
\end{proof}
\bigskip

We will now restrict to the case of $(1,f_0)$-jumping subspaces of an $(F_0,\mathcal{O}_X)$-Steiner bundle $E$ on $X$.

Following the notation set in Definition \ref{defjump}, we will denote the set of $(1,f_0)$-jumping subspaces by $J(E)$, the set of $(1,f_0)$-jumping pairs by $\tilde{J}(E)$, and by $\Sigma$ the set of $1$-dimensional subspaces $A\subset S^*$ for which there exists an $f_0$-dimensional subspace $B\subset H^0(F_0^\vee)$ such that $(A,B)$ is a $(1,f_0)$-jumping pair of $E$.

By abuse of notation, we will denote the jumping locus considered both as vectorial and projectivized. Therefore, the projectivization of the Segre embedding
$$
\nu : G(1,S^*) \times G(f_0,H^0(F_0^\vee)) \longrightarrow G(f_0,S^*\otimes H^0(F_0^\vee))
$$
is
$$
\nu : \PP(S) \times \GG\left(f_0-1,\PP\left(H^0(F_0^\vee)^*\right)\right) \longrightarrow \GG\left(f_0-1,\PP\left(S \otimes H^0(F_0^\vee)^*\right)\right)
$$
and it follows from the definition of jumping pair that
\begin{equation}\label{eq-jump}
\tilde{J}(E) = \image \nu \cap \GG(f_0-1,\PP(T_0)).
\end{equation}
Moreover, we can immediately obtain a lower bound for the dimension of $\tilde{J}(E)$ by computing the expected dimension of the intersection (the case when we have a complete intersection):
$$\dim \tilde{J}(E)\geq f_0\left(t_0-f_0+h^0(F_0^\vee)(1-s)\right)+s-1.$$

\begin{remark}
Observe that the previous inequality implies that the dimension of the jumping variety can be negative, which means that $\tilde{J}(E)$ can be empty. In this case the corresponding Steiner bundle $E$ cannot be a Schwarzenberger bundle.
\end{remark}

\section{The tangent spaces of the jumping variety}

Our main purpose in this section is to obtain an upper bound for the jumping pairs subspace of an $(F_0,\OO_X)$-Steiner bundle $E$. Our result will allow us to classify all Steiner bundles such that $\tilde{J}(E)$ has maximal dimension in the next section.

Consider a jumping pair $\Lambda$ in $\tilde{J}(E)$. Then $\Lambda$ is an $f_0$-dimensional vector space that can be written as $s_0 \otimes \Gamma \subset S^* \otimes H^0(F_0^\lor) \simeq \Hom (H^0(F_0^\lor)^*,S^*)$.
In addition, recall that the tangent space of the jumping variety at $\Lambda=s_0 \otimes \Gamma$ is the set
\[
\begin{array}{rl}
T_\Lambda \tilde{J}(E)= &\left\{ \psi\in \Hom\left(\Lambda,\frac{T^*}{\Lambda}\right)|\:\forall\: \varphi \in \Lambda, (\psi(\varphi))(\ker \varphi) \subset <s_0> \right.\\
&\left.  \: \mbox{and} \: \exists \: A \supset <s_0> \:\mbox{with} \:A \subset S^*, \dim A =2 \:\mbox{ such that}\: \image \psi(\varphi) \subset A\right\},
\end{array}
\]
as proved in \cite{AM}.

Since $\Lambda$ is a morphism in $\Hom (H^0(F_0^\lor)^*,S^*)$, we can construct three bases, $\{\lambda_i\}_{i=1}^{f_0}$ for $\Lambda$, $\{u_i\}_{i=1}^{N+1}$ for $H^0(F_0^\lor)^*$, with $N+1:=h^0(F_0^\vee)$ and $\{v_i\}_{i=1}^s$ for $S^*$, with $v_1=s_0$, such that
\[
\lambda_i: H^0(F_0^\vee)^*\to S^*
\]
$$
\lambda_i(u_j) = \left\{
\begin{array}{cl}
v_1 & \mbox{if}\:\: i=j,\\
0 & \mbox{if}\:\: i \neq j.
\end{array}
\right.
$$
In $\cite{AM}$ the authors proved that the tangent space of the jumping variety at a jumping pair can be also described as
\begin{equation}\label{eq-tang}
T_\Lambda \tilde{J}(E) = \left\{ \psi \in \Hom\left(\Lambda,\frac{T^*}{\Lambda}\right) \: | \:
\begin{array}{ll}
(\psi(\lambda_i))(\ker \lambda_i) \subset \langle v_1 \rangle,\\
(\psi(\lambda_i))(u_i) \equiv (\psi(\lambda_j))(u_j) \mod v_1, \,\,i\neq j
\end{array}
\right\}.
\end{equation}
Using this description, we are able to obtain an upper bound for the dimension of $T_\Lambda \tilde{J}(E)$ and hence an upper bound for $\tilde{J}(E)$.
\begin{thm}\label{thm-uppbound}
Let $E$ be a reduced $(F_0,\OO_X)$-Steiner bundle on $X$ and let $\sigma:X\to \GG(f_0-1, \PP(H_0(F_0^\lor)))$ be the natural morphism. For every $\Lambda \in \tilde{J}(E)$,
$$
\dim T_\Lambda \tilde{J}(E) \leq  f_0\left( t - \dim \sigma(X) - f_0 s +1\right).
$$
In particular, $\dim\tilde{J}(E) \leq  f_0\left( t - \dim \sigma(X) - f_0 s +1\right)$.
\end{thm}
\begin{proof}
Recall that the tangent space of the jumping variety at a jumping point is a vector subspace of  $\Hom\left(\Lambda,\frac{T^*}{\Lambda}\right)$. We will prove the result by defining independent elements in $\Hom\left(\Lambda,\frac{T^*}{\Lambda}\right)$ which are also independent modulo $T_\Lambda \tilde{J}(E)$. In order to do so, we will look for morphisms in $\frac{T^*}{\Lambda}$ which do not satisfy the conditions in (\ref{eq-tang}).\\

Let us consider the following diagram
$$
\xymatrix{
T^* \ar@{^{(}->}[r]^<<<<<{\varphi} \ar@{->>}[d] & \Hom\left(H^0(F_0^\lor)^*,S^*\right) \ar@{->>}[d] \\
\frac{T^*}{\Lambda} \ar[r]^<<<<<{\varphi_1} & \Hom\left(H^0(F_0^\lor)^*, \frac{S^*}{\langle v_1 \rangle} \right)
}
$$
Observe that the linear application $\varphi_1$ also defines an $(F_0,\OO_X)$-Steiner bundle $\tilde{E}$.

We want to estimate the dimension of the image of $\varphi_1$.
We first notice that the vector space $\Hom\left(H^0(F_0^\lor)^*, \frac{S^*}{\langle v_1 \rangle} \right)$ can be identified with the global sections of the bundle
$$
\frac{S}{\langle v_1 \rangle^*} \otimes \UU^\lor \longrightarrow \GG(f_0-1,\PP(H^0(F_0^\lor))).
$$
Then, the image of $\varphi_1$ can be therefore identified with the global sections of the restriction of the previous bundle to $\sigma(X)$, where $\sigma$ is the natural morphism $\sigma: X \longrightarrow \GG(f_0-1, \PP(H_0(F_0^\lor)))$.
Consider the bundle morphism
$$
\xymatrix{
\OO_{\sigma(X)}^\alpha \ar[rr]^g \ar[dr] & & \frac{S^*}{\langle v_1 \rangle} \otimes \UU^\lor \ar[ld]\\
& \sigma(X)
}
$$
By Porteous' formula, the morphism $g$ is not surjective if $\alpha \leq \dim \sigma(X) + (s-1)f_0-1$. We thus have at least $\dim \sigma(X) + (s-1)f_0$ independent morphisms in $\image\varphi_1$ (and hence $\mu_1,\ldots, \mu_{\dim \sigma(X) + (s-1)f_0}$ morphisms in $T^*/\Lambda$).

Let us construct morphisms $\psi_{i,j}$ belonging to $\Hom\left(\Lambda, \frac{T^*}{\Lambda}\right)$ in the following way:
\[
\psi_{i,j}: \Lambda \to T^*/ \Lambda
\]
\[
\psi_{i,j}(\lambda_k) = \left\{
\begin{array}{cl}
\mu_j & \mbox{if}\:\: i=t,\\
\lambda_t & \mbox{if}\:\: i \neq t,
\end{array}
\right. \]
for $i=2,\ldots,f_0$ and $j=1,\ldots,\dim \sigma(X) + (s-1)f_0$. Then we  have $(f_0-1)(\dim \sigma(X)+(s-1)f_0)$ of such independent morphisms and it can be easily verified that they do not satisfy any of the conditions in the definition of the tangent space.\\
Let us now extend the previous diagram to the following one:
$$
\xymatrix{
T^* \ar@{^{(}->}[r]^<<<<<{\varphi} \ar@{->>}[d] & \Hom\left(H^0(F_0^\lor)^*,S^*\right) \ar@{->>}[d]\\
\frac{T^*}{\Lambda} \ar[r]^<<<<<{\varphi_1} \ar[dr]^{\varphi_2} & \Hom\left(H^0(F_0^\lor)^*, \frac{S^*}{\langle v_1 \rangle} \right)\ar@{->>}[d] \\
& \Hom\left(\ker \lambda_1, \frac{S^*}{\langle v_1 \rangle}\right)
}
$$
We want to estimate the dimension of the image of $\varphi_2$.\\
As before, $\Hom\left(H^0(F_0^\lor)^*, \frac{S^*}{\langle v_1 \rangle} \right)$ can be identified with the global sections of the bundle
$$
\frac{S}{\langle v_1 \rangle^*} \otimes \UU^\lor \longrightarrow \GG(f_0-1,\PP(H^0(F_0^\lor))).
$$
and the image of $\varphi_1$ can be identified with the global sections of the restriction of the previous bundle to $\sigma(X)$.

Finally, if we take the subgrassmannian $\GG(f_0-1,\PP((\ker \lambda_1)^*))$, the image of $\varphi_2$ can be identified with the global sections of the vector bundle
$$
\frac{S}{\langle v_1 \rangle} \otimes \UU^\lor_{|\sigma(X) \cap \GG(f_0-1,\PP(\ker \lambda_1)^*))} \longrightarrow \sigma(X) \cap \GG(f_0-1,\PP(\ker \lambda_1)^*)).
$$
Denoting $W=\sigma(X) \cap \GG(f_0-1,\PP(\ker \lambda_1)^*))$, consider the bundle morphism
$$
\xymatrix{
\OO_W^\alpha \ar[rr]^g \ar[dr] & & \frac{S^*}{\langle v_1 \rangle} \otimes \UU^\lor_{|W} \ar[ld]\\
& W
}
$$
By Porteous' formula, the morphism $g$ is not surjective if $\alpha \leq \dim \sigma(X) + (s-2)f_0 -1$. This means that we can find at least $\dim \sigma(X) + (s-2)f_0$ independent morphisms in the image of $\varphi_2$ and hence $\tilde{\mu}_j\in T^*/\Lambda$, with $j=1,\ldots, \dim \sigma(X) + (s-2)f_0$. Define a new set $\{\psi_{1,j}\}_{j=1}^{\dim \sigma(X) + (s-2)f_0} \subset \Hom\left(\Lambda, \frac{T^*}{\Lambda}\right)$, by
\[
\psi_{1,j}: \Lambda \to T^*/ \Lambda
\]
\[
\psi_{1,j}(\lambda_k) = \left\{
\begin{array}{cl}
\tilde{\mu}_j & \mbox{if}\:\: k=1,\\
\lambda_t & \mbox{if}\:\: k \neq 1,
\end{array}
\right. \]
and not satisfying the conditions defining the tangent space.

The morphisms $\psi_{i,j}$ just constructed are linearly independent elements in the complementary vector space of $T_\Lambda \tilde{J}(E)$ and so the theorem is proved.
\end{proof}

\section{The classification}
The main goal in this section is to classify all $(F_0,\OO_X)$-Steiner bundles whose jumping variety has maximal dimension. In particular, we prove that they are always Schwarzenberger bundles.

Recall (see (\ref{vbK})) that if $E$ is an $(F_0,\OO_X)$-Steiner bundle on a smooth projective variety $X$ we have an the exact sequence given by
$$
0 \longrightarrow Q \longrightarrow H^0(F_0^\lor) \otimes \OO_X \longrightarrow F_0^\lor \longrightarrow 0.
$$

Let us first state the theorem. The rest of the section will be devoted to its proof.
\begin{thm}\label{thm-classif}
Let $E$ be a reduced $(F_0,\OO_X)$-Steiner bundle on a smooth projective variety $X$ such that the jumping locus $\tilde{J}(E)$ has maximal dimension. Suppose that the morphism $\sigma: X \longrightarrow \GG(f_0-1,\PP(H^0(F_0^\lor)))$ is generically finite.

If $\dim X \geq s\cdot \rk(Q)$, then $E \simeq S \otimes Q^\lor$.

If $\dim X < s \cdot\rk(Q)$, then $E$ is one of the following:
\begin{description}
\item{i)} a Schwarzenberger bundle given by the triple
$$\left(\tilde{J}(E), |\pi_2^*(\OO_{\PP^N}(1))|, \pi_1^*(\OO_{\PP^1}(s-1))\right).$$
In this case $f_0=1$, $\tilde{J}(E)$ is a rational normal curve and we have the natural projections
$$
\begin{array}{l}
\tilde{J}(E) \stackrel{\pi_1}{\longrightarrow}  \Sigma(E) \simeq \PP^1 \vspace{2mm}\\
\tilde{J}(E) \stackrel{\pi_2}{\longrightarrow}  \PP^N
\end{array}
$$
with $N=h^0(F_0^\lor)-1$;\\
\item
{ii)} a Schwarzenberger bundle given by the triple
$$\left(\tilde{J}(E), |\pi_2^*(\UU^\lor)|, \pi_1^*(\OO_{\PP(S)}(1))\right).$$
In this case $s \leq f_0+1$, $\tilde{J}(E)$ is the projectivization of a Grassmannian bundle constructed from a rational normal scroll and we have the natural projections
$$
\begin{array}{l}
\tilde{J}(E) \stackrel{\pi_1}{\longrightarrow}  \Sigma(E) \simeq \PP(S) \vspace{2mm}\\
\tilde{J}(E) \stackrel{\pi_2}{\longrightarrow}  \GG\left(f_0-1,\PP(H^0(F_0^\lor))\right);
\end{array}
$$
\item
{iii)} a Schwarzenberger bundle given by the triple
$$\left(\tilde{J}(E), |\pi_2^*(\UU^\lor)|, \OO_{\tilde{J}(E)}(1))\right).$$
In this case $f_0>1$, $\tilde{J}(E) \simeq \Sigma(E)$ and we have the natural projection
$$
\xymatrix{
 \tilde{J}(E) \ar[r]^<<<<<{\pi_2} & \GG\left(f_0-1,\PP(H^0(F_0^\lor))\right);
}
$$
\item{iv)} a Schwarzenberger bundle given by the triple
$$\left(\tilde{J}(E), |\pi_2^*(\OO_{\PP^1}(1))|, \pi_1^*(\OO_{\Sigma(E)}(1))\right).$$
In this case $f_0=1$, $s\geq 3$, $\tilde{J}(E)$ is a rational normal scroll and we have the natural projections
$$
\begin{array}{l}
\tilde{J}(E) \stackrel{\pi_1}{\longrightarrow}  \Sigma(E)  \vspace{2mm}\\
\tilde{J}(E) \stackrel{\pi_2}{\longrightarrow}  J(E) \simeq \PP^1;
\end{array}
$$
\item{v)} a Schwarzenberger bundle given by the triple
$$\left(\tilde{J}(E), |\pi_2^*(\OO_{\PP^2}(1))|, \pi_1^*(\OO_{\PP^2}(1))\right).$$
In this case $f_0=1$, $s=3$, $\tilde{J}(E)$ is a Veronese surface and we have the natural projections
$$
\begin{array}{l}
\tilde{J}(E) \stackrel{\pi_1}{\longrightarrow}  \Sigma(E)\simeq \PP^2 \vspace{2mm}\\
\tilde{J}(E) \stackrel{\pi_2}{\longrightarrow}  J(E) \simeq \PP^2.
\end{array}
$$
\end{description}
\end{thm}
\begin{remark}
If the morphism $\sigma: X \longrightarrow \GG(f_0-1,\PP(H^0(F_0^\lor)))$ is surjective, then an $(F_0,\OO_X)$-Steiner bundle $E$ on $X$ induces a  $(\UU,\OO_\GG)$-Steiner bundle $\bar{E}$ on the Grassmannian, with $E = \sigma^*(\bar{E})$. Moreover, $\tilde{J}(E)$ has maximal dimension if and only if $\tilde{J}(\bar{E})$ has maximal dimension, according to the respective bounds. Therefore, when $\sigma$ is surjective and $\tilde{J}(E)$ has maximal dimension, all Steiner bundles on $X$ are Schwarzenberger bundles given by the pullback of the corresponding Schwarzenberger bundle on $\GG(f_0-1,\PP(H^0(F_0^\lor)))$, classified in \cite{AM}. Therefore, we can suppose, from now on, that $\sigma$ is not surjective
\end{remark}

Let $E$ be a reduced $(F_0,\OO_X)$-Steiner bundle.

The first statement in the theorem, if $\dim X \geq s \cdot\rk(Q)$, follows from Proposition \ref{prop-trivial} when $E$ is reduced (recall also Remark \ref{gen-fin}).

Suppose now that $\dim X < s \cdot\rk(Q)$.
Take $\Lambda = s_0 \otimes \Gamma \in \tilde{J}(E)$. We have the
following commutative diagram:
\begin{equation}\label{diag-ind}
\xymatrix{
T^* \ar@{^{(}->}[r]^<<<<<{\varphi} \ar[d] & S^* \otimes H^0(F_0^\lor) \ar[d]^{pr\otimes id} \\
\frac{T^*}{\Lambda} \ar[r]^<<<<<{\varphi'} & \frac{S^*}{\langle s_0 \rangle} \otimes H^0(F_0^\lor)
}
\end{equation}
The morphism $\varphi'$ defines a new $(F_0,\OO_X)$-Steiner bundle $E'$, whose defining vector spaces $S'$ and $T'$ have dimension $s-1$ and $t-f_0$, respectively. Let $E'_0$ be its reduced summand.
Iterating this process, which we will call \emph{induction technique}, we will see that we always eventually arrive to a known case of a reduced $(F_0,\OO_X)$-Steiner bundle that can be described as a Schwarzenberger bundle.

Moreover, we also have a diagram
\begin{equation}\label{diag-clas}
\xymatrix{
& \tilde{J}(E) \ar[dl]_{\pi_1} \ar[d] \\
\Sigma(E) \ar[d]_{pr_{s_0}} & \tilde{J}(E'_0) \ar[dl]_{\pi_1'} \\
\Sigma(E_0')
}
\end{equation}
Properties in Theorems 4.3 and 4.4 in \cite{AM} will still hold. In particular, $\Sigma(E)$ will always be a minimal degree variety,  $\tilde{J}(E_0')$ is birational to $\Sigma(E)$ and the morphism $pr_{s_0}$ is a projection from an inner point $s_0 \in \Sigma(E)$, hence $\dim \Sigma(E'_0) \leq \dim \Sigma(E) \leq \dim \Sigma(E'_0) +1$. 

The generic fibers of $\pi_1$, $\pi_1'$, and the further first component projections given by the induction technique, have dimension either 0 or at least $f_0$.

\subsection{The case $s \leq f_0+1$}

Suppose that $s \leq f_0+1$ and consider an $(F_0, \OO_X)$-Steiner bundle $E$ on $X$ whose jumping variety $\tilde{J}(E)$ has maximal dimension, i.e.
$$
\dim \tilde{J}(E) = f_0\left(t-\dim X - s f_0 +1 \right).
$$
\begin{prop}
If $s \leq f_0+1$ then the projection $\pi_1: \tilde{J}(E) \longrightarrow \PP(S)$ is surjective.
\end{prop}
\begin{proof}
Suppose that $\pi_1$ is not surjective, which implies that $\dim \Sigma(E) < s -1\leq f_0 $. Hence, for each $s_0 \in \Sigma(E)$, we have that $\dim \pi_1^{-1}(s_0) =\dim\tilde{J}(E)-\dim\Sigma(E)> f_0(t-\dim X -  s f_0 )$. Since $\pi_1^{-1}(s_0) \simeq G(f_0,\left(\langle s_0 \rangle \otimes H^0(F_0^\lor) \right) \cap T^*)$, we get that $\dim \left(\left(\langle s_0 \rangle \otimes H^0(F_0^\lor) \right) \cap T^* \right) > t - \dim X + (s-1)f_0$. Consider the following diagram, obtained by taking a jumping pair $\Lambda = s_0 \otimes \Gamma \in \tilde{J}(E)$:
$$
\xymatrix{
 & 0 \ar[d] \\
 & \langle s_0 \rangle \otimes H^0(F_0^\lor) \ar[d] \\
 T^* \ar@{->>}[d]_{\pi} \ar[ur]^<<<<<<<<{\varphi''} \ar@{^{(}->}[r]^<<<<<<{\varphi} & S^* \otimes H^0(F_0^\lor) \ar[d] \\
 \frac{T^*}{\Lambda} \ar[r]^<<<<<<{\varphi'} & \frac{S^*}{\langle s_0 \rangle} \otimes H^0(F_0^\lor) \ar[d] \\
 & 0
}
$$
Being $\image \: \varphi'' \simeq \left(\left(\langle s_0 \rangle \otimes H^0(F_0^\lor) \right) \cap T^* \right)$,
we have $\dim \image \varphi' =\dim \image \varphi -\dim \image \varphi'' < \dim X +(s-1)f_0$.

Suppose first $\dim X<(s-1)\cdot\rk(Q)$. Since $\varphi '$ also defines a Steiner bundle, it follows from Proposition \ref{thm-rank} that $\dim \image \varphi' \geq \dim X +(s-1)f_0$, leading us to contradiction.

Otherwise, when $\dim X\geq (s-1)\cdot\rk(Q)$, we get that $E'_0\simeq (S^* / \langle s_0)^*\rangle\otimes Q$. Therefore, $\tilde{E'_0}$ is the whole Segre variety (recall (\ref{eq-jump})) and $\Sigma(E'_0)=\PP\left((S^* / \langle s_0)^*\right)$. Then $\dim\Sigma(E)$ must be equal to $s-2$ and $\Sigma(E)$ is birational to $\Sigma(E'_0)$. Being $\Sigma(E)$ also birational to $\tilde{J}(E'_0)$, this leads to a contradiction.
\end{proof}
\begin{remark}
Except when $E$ is a trivial Steiner bundle, one deduces from the previous proposition that under the given hypothesis all fibers $\pi_1^{-1}(s_0)$, for each $s_0 \in \PP(S)$, have the same dimension, namely $f_0(t-\dim X - sf_0 +1) - (s-1)$. Hence, when $s\leq f_0+1$, the jumping variety $\tilde{J}(F)$ is the projectivization of a Grassmannian bundle constructed from a rational scroll on $\PP(S)$. In particular, it is smooth.
\end{remark}

Consider the two natural projections $\tilde{J}(E) \stackrel{\pi_1}{\longrightarrow} \PP(S)$ and $\tilde{J}(E) \stackrel{\pi_2}{\longrightarrow} \GG(f_0-1,\PP(H^0(F_0^\lor)^*))$. By Lemma \ref{lem-class} we obtain the following result.
\begin{thm}\label{basecase}
Let $E$ be a reduced $(F_0,\OO_X)$-Steiner bundle on a smooth projective variety $X$ such that  $\tilde{J}(E)$ has maximal dimension and $\sigma: X \longrightarrow \GG(f_0-1,\PP(H^0(F_0^\lor)))$ is generically finite. If $s\leq f_0+1$ then $E$ is a Schwarzenberger bundle defined by the triple $$(\tilde{J}(E), | \pi_2^*(\UU^\lor)|, \pi_1^*(\OO_{\PP(S)}(1))),$$ where $\UU \longrightarrow \GG(f_0-1,\PP(H^0(F_0^\lor)^*))$ denotes the universal bundle of rank $f_0$.
\end{thm}

This proves part $(ii)$ of Theorem \ref{thm-classif}.

\subsection{The general case}
Recall diagram (\ref{diag-clas}) and the fact that each first component projection has dimension either $0$ or at least $f_0$. We will thus divide the classification in the cases $f_0=1$ and $f_0>1$.

\vspace{3mm}

\noindent\textbf{Case $f_0 = 1$}{\ }

If $f_0 =1$ and $s=2$ we have already proved in Theorem \ref{basecase} that $\tilde{J}(E)$ is a rational normal scroll, $\Sigma(E)=\PP^1$ and, moreover, $E$ is Schwarzenberger bundle.

It was proved in \cite{AM} that if $\pi_1$ is not birational then each further projection on the first component, given by the iteration process, is not birational.

Let us consider first the case when all projections $\pi_1$ are birational. Applying the induction technique until the case $s=2$, we obtain that $\tilde{J}(E)$ is a rational normal curve because it is birational to $\PP^1$. So, considering the following diagram
$$
\xymatrix{
\tilde{J}(E)\ar[d]_{\pi_1} \ar[r]^<<<<<{\pi_2} & \PP(H^0(F_0^\lor))=: \PP^N\\
\PP^1
}
$$
we get, by Lemma \ref{lem-class}, that $E$ is Schwarzenberger bundle on $X$ defined by the triple $$\left(\tilde{J}(E), |\pi_2^*(\OO_{\PP^N}(1))|, \pi_1^*(\OO_{\PP^1}(s-1))\right).$$

Let us suppose now that the birationality is broken at some step of the induction. Without loss of generality, we can focus on such step and we are in the following situation (+1 will denote that the fiber is one dimensional):
$$
\xymatrix{
& & \tilde{J}(E) \ar[dl]_{bir} \ar[d]^{bir}\\
\mbox{step} \,\, s &\Sigma(E) \ar[d]_{+1} \ar@{--}[r]^{bir} & \tilde{J}(E_0') \ar@{-->}[dd] \ar[dl]_{+1}\\
\mbox{step} \,\, s-1 & \PP^{s-2} \simeq \Sigma(E_0')\ar@{-->}[dd] & \\
& & \tilde{J}(\bar{E}_0)\ar[dl]_{+1}\\
\mbox{step} \,\, 2 & \PP^1
}
$$
From the classification of the case $s=2$, we have that $\tilde{J}(\bar{E}_0)$ is a surface. One can prove that it is a quadric surface, which implies that $J(\bar{E})\simeq\PP^1$. Therefore, also $J(E)\simeq \PP^1$ and $\tilde{J}(E)$ is a rational normal scroll over $\PP^1$. Denoting as usual  $\pi_1:\tilde{J}(E)\longrightarrow \Sigma(E)$ and $
\pi_2:\tilde{J}(E)\longrightarrow\PP^1$,
we obtain by Lemma \ref{lem-class} that, if $s\geq 4$, $E$ is the Schwarzenberger bundle defined by the triple $$\left(\tilde{J}(E), |\pi_2^*(\OO_{\PP^1}(1))|, \pi_1^*(\OO_{\Sigma(E)}(1))\right).$$

Let us consider the only left case when $f_0=1$, that is the case $s=3$, described by the following diagram:
$$
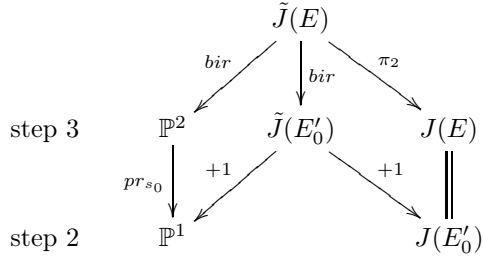
\xymatrix{
&& \tilde{J}(E) \ar[dl]_{bir} \ar[d]^{bir} \ar[dr]^{\pi_2} \\
\mbox{step} \,\, 3 &\PP^2 \ar[d]_{pr_{s_0}} & \tilde{J}(E_0') \ar[dl]_{+1} \ar[dr]^{+1} & J(E) \ar@{=}[d]\\
\mbox{step} \,\, 2 &\PP^1 & & J(E_0')
}
$$

We have already proved that $\tilde{J}(E_0')$ is a rational normal scroll of dimension 2, and, as shown in \cite{AM}, the variety $\tilde{J}(E)$ can be either a Hirzebruch surface or a Veronese surface.\\
In the first case $E$ is the Schwarzenberger bundle given by the triple
$$\left(\tilde{J}(E), |\pi_2^*(\OO_{\PP^1}(1))|, \pi_1^*(\OO_{\Sigma(E)}(1))\right).$$
In the second case we have that $E$ is the Schwarzenberger bundle given by the triple
$$
\left(\tilde{J}(E), |\pi_2^*(\OO_{\PP^2})|, \pi_1^*(\OO_{\PP^2}(1))\right).
$$

This proves parts $(i)$, $(iv)$ and $(v)$ of Theorem \ref{thm-classif}.
\vspace{3mm}\\
\textbf{Case $f_0 > 1$}{\ }

Let us study now the case with $f_0 >1$. Looking at the Diagram (\ref{diag-clas}), we recall it is impossible to get $\dim \Sigma(E) = \dim \Sigma(E_0') +1$, because we have already noticed that the fiber of the projections of type $\pi_1$ has dimension either zero or greater equal than $f_0$, which would lead to contradiction. This means that all the projections involved in the diagram are birational. Recall the following lemma proved in \cite{AM}, that also applies in the current situation.
\begin{lemma}
Let $E$ be an $(F_0,\OO_X)$-Steiner bundle on a smooth projective variety $X$ and $\tilde{J}(E)$ its jumping locus. Suppose that $\tilde{J}(E)$ is birational to $\Sigma(E)$ and, fixed a jumping pair $s_0\otimes \Gamma$, consider the first step of the induction (Diagram (\ref{diag-clas})). If the morphism $\pi_1'$ is an isomorphism then also $\pi_1$ is an isomorphism.
\end{lemma}
The combination of this lemma and the birationality of the projections imply that $\tilde{J}(E) \simeq \Sigma(E)$, and so $E$ is the Schwarzenberger bundle given by the triple
$$
\left(\tilde{J}(E), |\pi_2^*(\UU^\lor)|, \OO_{\tilde{J}(E)}(1)\right),
$$
proving $(iii)$ in the theorem.

\bigskip

We have completed the study of all possible cases  and we have thus proved Theorem \ref{thm-classif}.

%

\bibliographystyle{alpha}
\bibliography{AMSSchwsmooth}

\end{document}